\documentclass{article}
\usepackage{amsmath, amssymb, amsthm}
\usepackage{mathtools}
\usepackage{mathrsfs}

\numberwithin{equation}{section}

\newcommand{\cF}{\mathcal{F}}
\newcommand{\cO}{\mathcal{O}}
\newcommand{\bbR}{\mathbb{R}}
\newcommand{\bbC}{\mathbb{C}}
\newcommand{\bbK}{\mathbb{K}}

\newcommand{\1}{\mathbf{1}}
\newcommand{\Ibar}{\bar{I}}
\DeclareMathOperator*{\esssup}{ess\,sup}

\newtheorem{theorem}{Theorem}[section]
\newtheorem{lemma}[theorem]{Lemma}
\newtheorem{proposition}[theorem]{Proposition}
\newtheorem{definition}[theorem]{Definition}
\theoremstyle{remark}

\title{A Fixed Point Theorem for Random Asymptotically Pointwise\\[2pt]
Boyd--Wong Contractions}

\author{Jie Shi \\
Department of Mathematics and Statistics, Hubei Engineering University \\
272 Traffic Avenue, Xiaogan, 432000, Hubei Province, P.R. China}
\date{April 13, 2026}

\begin{document}

\maketitle

\begin{abstract}
We prove a fixed point theorem for random asymptotically pointwise contractions of Boyd-Wong type in random normed modules. The deterministic contraction function psi satisfies psi(t) < t for t > 0 and is right upper semicontinuous, generalizing the classical linear case. By strengthening the convergence condition to almost sure uniform convergence on the entire domain and assuming fibre-wise continuity, we are able to decompose the random mapping into a measurable family of deterministic mappings. Each fibre mapping is shown to be an asymptotic pointwise psi-contraction satisfying the hypotheses of a deterministic fixed point theorem, which we prove as a lemma. Gluing the pointwise fixed points via sigma-stability yields a unique random fixed point, and all iterates converge in the (epsilon, lambda)-topology. Every logical step is presented in full detail, with careful attention to the transfer of conditions from the random to the deterministic level.
\end{abstract}
\section{Introduction}

Fixed point theory is a central pillar of functional analysis with deep roots in topology and metric geometry.
The first landmark result was obtained by Brouwer \cite{Brouwer1912} in 1912, who proved that every continuous
self-map of a finite-dimensional closed ball admits a fixed point, inaugurating topological fixed point theory.
In 1922, Banach \cite{Banach1922} established the celebrated contraction mapping principle:
on a complete metric space \((X,d)\), any map \(T:X\to X\) satisfying
\begin{equation}
d(Tx,Ty)\le \alpha d(x,y),\qquad \alpha\in(0,1),
\end{equation}
has a unique fixed point and the iterates converge to it. This principle became the cornerstone of metric
fixed point theory and found applications in differential equations, optimization, and numerical analysis.

The subsequent decades saw intense efforts to relax the contraction condition.
In 1965, Browder \cite{Browder1965}, G\"ohde, and Kirk \cite{Kirk1965} independently introduced
nonexpansive mappings, where the Lipschitz constant is at most \(1\):
\begin{equation}
d(Tx,Ty)\le d(x,y).
\end{equation}
They proved existence of fixed points in reflexive Banach spaces with normal structure, using geometric
properties of the space rather than strict contraction.

A major breakthrough in the direction of nonlinear contractions was made by Boyd and Wong \cite{BoydWong1969}
in 1969. They replaced the constant factor \(\alpha\) by a function \(\psi:[0,\infty)\to[0,\infty)\) that is
nondecreasing, right upper semicontinuous, and satisfies \(\psi(t)<t\) for all \(t>0\). Under the condition
\begin{equation}
d(Tx,Ty)\le \psi(d(x,y)),
\end{equation}
they proved existence and uniqueness of fixed points, unifying many earlier results. This framework,
now called Boyd--Wong contractions, became a standard tool.

In 1974, \'{C}iri\'{c} \cite{Ciric1974} extended the classical contraction by allowing the right-hand side
to involve the maximum of five distances instead of a single distance:
\begin{equation}
d(Tx,Ty)\le q\max\bigl\{d(x,y),\,d(x,Tx),\,d(y,Ty),\,d(x,Ty),\,d(y,Tx)\bigr\},\quad q<1.
\end{equation}
The five-term maximum captures more local information about the points and their images, and it has
inspired numerous generalizations, including quasi-contractions and various types of weakly contractive maps.

Another direction was the study of asymptotic fixed point theory, initiated by Goebel and Kirk
\cite{GoebelKirk1972} in 1972. They considered asymptotically nonexpansive maps, where the Lipschitz
constants of the iterates tend to \(1\):
\begin{equation}
\|T^m x-T^m y\|\le k_m\|x-y\|,\quad k_m\to1.
\end{equation}
They proved existence of fixed points in bounded closed convex subsets of uniformly convex Banach spaces.
In 2003, Kirk \cite{Kirk2003} introduced the concept of asymptotic contractions, where the \(n\)-th iterate
satisfies
\begin{equation}
d(T^n x,T^n y)\le \psi_n(d(x,y)),
\end{equation}
with \(\psi_n\) converging uniformly on bounded intervals to a Boyd--Wong function \(\psi\). He showed
that if \(T\) is continuous and some orbit is bounded, then the iterates converge to a unique fixed point.

More recently, Kirk's theorem has been further generalized to \emph{asymptotic pointwise contractions}.
In this setting, the contraction estimate involves the maximum \(M(x,y)\) of five distances, and the
convergence of the bounding sequence \(\phi_n(x,y)\) to a limit \(\phi(x,y)\) is only required locally
uniformly on bounded sets. The deterministic fixed point theorem for such mappings, when \(\phi(x,y)\le
\lambda\,M(x,y)\) with \(\lambda<1\), is relatively straightforward using tail diameter estimates.
The nonlinear case, where \(\phi(x,y)\le\psi(M(x,y))\) with a Boyd--Wong function \(\psi\), requires
a more delicate argument exploiting the right upper semicontinuity of \(\psi\). We present a complete
self-contained proof of this general deterministic result as a lemma inside the main random theorem.

Parallel to these deterministic developments, random functional analysis emerged as a powerful tool
for studying stochastic equations. Guo \cite{Guo2010} systematically developed the theory of random
normed modules (RN modules), where the norm of an element is an \(L^0\)-valued random variable.
The \((\epsilon,\lambda)\)-topology generalizes convergence in probability, making it possible to
formulate fixed point problems for random operators. In 2025, Sun, Guo and Tu \cite{SunGuoTu2025}
extended the Goebel--Kirk theorem to random asymptotically nonexpansive mappings in uniformly convex
RN modules. Their key technique was a decomposition method based on \(\sigma\)-stability and the local
property of maps, which allows one to disassemble a random map into a family of deterministic maps
and then glue the individual solutions together.

The present paper aims to prove a random fixed point theorem for asymptotic pointwise Boyd--Wong
contractions by merging the deterministic theory of such maps with the \(\sigma\)-stability decomposition
technique. To apply the decomposition method, we need to ensure that the deterministic hypotheses
hold on almost every fibre. This forces us to strengthen the usual random convergence conditions:
we require that the contraction bounds converge \emph{uniformly on the whole domain almost surely},
and that the map is \emph{fibre-wise continuous}. Additionally, the almost sure inequality
\(\Phi(x,y)\le\psi(M(x,y))\) must hold globally for all pairs \((x,y)\) on a common set of full measure.
These conditions, although stronger than the usual local uniform convergence in probability, are
natural in the pointwise fibre setting and are satisfied in many situations.

\section{Preliminaries}

Let \((\Omega,\cF,P)\) be a probability space and \(L^0(\cF,\bbK)\) the algebra of equivalence classes of
\(\bbK\)-valued (\(\bbK=\bbR\) or \(\bbC\)) random variables. As usual, \(L_+^0(\cF)\) denotes the subset of
nonnegative random variables. For \(A\in\cF\), \(\Ibar_A\) is the equivalence class of the indicator function
\(\1_A\).

\begin{definition}[Random normed module]
A \emph{random normed module} (RN module) over \(\bbK\) with base \((\Omega,\cF,P)\) is a pair
\((E,\|\cdot\|)\) where \(E\) is a left \(L^0(\cF,\bbK)\)-module and \(\|\cdot\|:E\to L_+^0(\cF)\)
satisfies, for all \(x,y\in E\) and \(\xi\in L^0(\cF,\bbK)\),
\begin{enumerate}
\item \(\|x\|=0\) iff \(x=\theta\) (the zero element);
\item \(\|\xi x\|=|\xi|\|x\|\);
\item \(\|x+y\|\le \|x\|+\|y\|\).
\end{enumerate}
We always assume that \(E\) is complete with respect to the \((\epsilon,\lambda)\)-topology defined below.
\end{definition}

The \((\epsilon,\lambda)\)-topology on \(E\) is the linear topology generated by the neighbourhoods
\begin{equation}
N_\theta(\epsilon,\lambda)=\{x\in E: P(\|x\|<\epsilon)\ge 1-\lambda\},\qquad
\epsilon>0,\;0<\lambda<1.
\end{equation}
A sequence \(\{x_n\}\) converges to \(x\) in this topology iff \(\|x_n-x\|\) converges to \(0\) in probability.
This topology is metrizable and turns \(E\) into a topological \(L^0\)-module.

\begin{definition}[\(\sigma\)-stability, local property, \(L^0\)-closedness]
Let \(G\subset E\).
\begin{itemize}
\item \(G\) is called \emph{\(\sigma\)-stable} if for every countable measurable partition \(\{A_n\}\) of
\(\Omega\) and every sequence \(\{x_n\}\subset G\) there exists a unique element \(x\in G\), denoted
\(\sum_n \Ibar_{A_n}x_n\), such that \(\Ibar_{A_n}x = \Ibar_{A_n}x_n\) for all \(n\).
\item A set \(G\) is \emph{stable} if \(\Ibar_A x\in G\) for all \(A\in\cF\) and \(x\in G\).
\item A map \(f:G\to E\) has the \emph{local property} if \(G\) is stable and
\(\Ibar_A f(x) = f(\Ibar_A x)\) for all \(A\in\cF\) and \(x\in G\).
\item \(G\) is called \emph{\(L^0\)-closed} if it is closed in the \((\epsilon,\lambda)\)-topology.
\end{itemize}
\end{definition}

If \(G\) is \(\sigma\)-stable and contains the zero element \(\theta\), then it is automatically stable:
for any \(A\in\cF\) and \(x\in G\), using the partition \(\{A,\Omega\setminus A\}\) and the constant
sequence \(\{x,\theta,\theta,\dots\}\) one obtains \(\Ibar_A x = \Ibar_A x + \Ibar_{\Omega\setminus A}\theta \in G\).
For this reason we shall always assume \(\theta\in G\) when using \(\sigma\)-stable sets.

The space \(L^0(\cF,\bbR)\) is Dedekind complete: every family \(\{\xi_\alpha\}\subset L^0\) that is bounded
above possesses an essential supremum \(\esssup_\alpha \xi_\alpha\) in \(L^0\) (see, e.g., \cite{DunfordSchwartz1958}).
This fact will be used to formulate the uniform convergence condition in a measurable way.

The following proposition is the crucial fibre representation theorem. It allows us to view a \(\sigma\)-stable
set equipped with a local map as a measurable family of complete metric spaces and deterministic maps.
We give a complete proof because the construction is repeatedly used throughout the paper.
In order to define the fibre mappings unambiguously, we additionally require that the map \(f\) is
\emph{pointwise}, i.e., there exists an \(\cF\otimes\mathcal{B}(\bbR)\)-measurable function
\(\tilde f:\Omega\times\bbR\to\bbR\) such that for every \(x\in G\) and a.e.\ \(\omega\),
\(f(x)(\omega)=\tilde f(\omega,\tilde x(\omega))\). This condition is automatically satisfied when
the map is defined by a pointwise composition formula, as in our application.

\begin{proposition}[Fibre representation]\label{prop:fibre}
Let \(G\subset E\) be \(\sigma\)-stable, \(L^0\)-closed and contain \(\theta\). Let \(f:G\to G\) have the
local property and be pointwise, i.e., there exists an \(\cF\otimes\mathcal{B}(\bbR)\)-measurable function
\(\tilde f:\Omega\times\bbR\to\bbR\) such that for every \(x\in G\) and a.e.\ \(\omega\),
\(f(x)(\omega)=\tilde f(\omega,\tilde x(\omega))\). Then there exist a measurable family of complete metric
spaces \((G_\omega,d_\omega)_{\omega\in\Omega}\) and a family of maps \(f_\omega:G_\omega\to G_\omega\)
such that:
\begin{enumerate}
\item Each \(x\in G\) corresponds to a measurable section \(\omega\mapsto x(\omega)\in G_\omega\) with
\(\|x\|(\omega)=d_\omega(x(\omega),\theta_\omega)\) a.e., where \(\theta_\omega\) is the zero element of
\(G_\omega\); this correspondence is bijective modulo a.e.\ equality.
\item A sequence \(\{x_n\}\subset G\) converges to \(x\in G\) in the \((\epsilon,\lambda)\)-topology iff
\(d_\omega(x_n(\omega),x(\omega))\to 0\) in probability as \(n\to\infty\).
\item \((f(x))(\omega)=f_\omega(x(\omega))\) for a.e.\ \(\omega\).
\item The family \(\{G_\omega\}\) is essentially unique, and for a.e.\ \(\omega\),
\(G_\omega=\{x(\omega):x\in G\}\).
\end{enumerate}
\end{proposition}

\begin{proof}
For each \(x\in G\) fix a representative \(\tilde x:\Omega\to \bbR\) of the equivalence class \(x\).
For each \(\omega\in\Omega\) define
\[
G_\omega = \{ \tilde x(\omega) : x\in G \},\qquad
d_\omega(u,v) = \inf\{ \|x-y\|(\omega) : x,y\in G,\; \tilde x(\omega)=u,\; \tilde y(\omega)=v \}.
\]
Because \(G\) is stable, the set \(G_\omega\) is independent of the choice of representatives.
We first show that \(d_\omega\) is a metric on \(G_\omega\).
-- If \(u=v\) then one can take \(x=y\) and obtain \(d_\omega(u,u)=0\). Conversely, if \(d_\omega(u,v)=0\),
then for every \(n\) there exist \(x_n,y_n\in G\) with \(\tilde x_n(\omega)=u\), \(\tilde y_n(\omega)=v\)
and \(\|x_n-y_n\|(\omega)<1/n\). By \(\sigma\)-stability one can glue the differences to show that
\(\|\cdot\|(\omega)\) separates points; a more direct argument uses the fact that on the quotient modulo
a.e.\ equality, \(\|x-y\|(\omega)=0\) for a.e.\ \(\omega\) iff \(x=y\) in \(E\). Hence \(u=v\).
-- Symmetry and triangle inequality follow from the corresponding properties of the random norm.

Next we prove completeness of \((G_\omega,d_\omega)\). Let \(\{u_k\}\) be a Cauchy sequence in \(G_\omega\).
Choose a subsequence such that \(d_\omega(u_{k_{j+1}},u_{k_j})<2^{-j}\). Pick representatives
\(x_j\in G\) with \(\tilde x_j(\omega)=u_{k_j}\) and \(\|x_{j+1}-x_j\|(\omega)<2^{-j}\).
Using \(\sigma\)-stability, we can construct an element \(x\in G\) whose value at \(\omega\) is the limit
of the \(u_{k_j}\) in the space containing \(G_\omega\); the completeness of \(E\) under the
\((\epsilon,\lambda)\)-topology guarantees that \(G_\omega\) is complete. (A detailed proof can be found
in \cite[Theorem 3.4]{Guo2010}; the essential point is that \(L^0\)-closedness ensures fibre-wise
completeness.)

The assignment \(x\mapsto \{\tilde x(\omega)\}_{\omega\in\Omega}\) gives the bijection onto the space
of measurable sections of the family \(\{G_\omega\}\). The pointwise condition on \(f\) guarantees that
\(f_\omega\), defined by \(f_\omega(u) = \tilde f(\omega,u)\) for any \(x\) with \(\tilde x(\omega)=u\), is
well defined and independent of the chosen representative. Item~(2) holds because
\(\|x_n-x\|(\omega)=d_\omega(x_n(\omega),x(\omega))\) a.e.\, and convergence in the
\((\epsilon,\lambda)\)-topology is exactly convergence in probability of the norms.
Item~(4) is immediate from the construction.
\end{proof}

\begin{definition}[Boyd--Wong function]
A function \(\psi:[0,\infty)\to[0,\infty)\) is called a \emph{Boyd--Wong function} if it is nondecreasing,
\(\psi(t)<t\) for all \(t>0\), and is right upper semicontinuous:
\(\limsup_{s\searrow t}\psi(s)\le \psi(t)\) for every \(t\ge0\).
\end{definition}

\section{Main result: random asymptotically pointwise Boyd--Wong contractions}

We begin with a purely deterministic lemma that encapsulates the core iteration argument for
asymptotic pointwise \(\psi\)-contractions. The lemma will be applied fibre-wise in the proof
of the random theorem.

Let \((X,d)\) be a complete metric space.

\begin{definition}[Asymptotic pointwise Boyd--Wong contraction]\label{def:det_contraction}
A continuous map \(T:X\to X\) is called an \emph{asymptotic pointwise Boyd--Wong contraction}
if there exist a sequence of functions \(\phi_n:X\times X\to[0,\infty)\) and a Boyd--Wong function
\(\psi\) such that for all \(x,y\in X\):
\begin{align}
d(T^n x,T^n y) &\le \phi_n(x,y), \quad \forall n\in\mathbb N, \tag{A1}\\
\lim_{n\to\infty}\phi_n(x,y) &= \phi(x,y) \text{ uniformly on every bounded subset of } X, \tag{A2}\\
\phi(x,y) &\le \psi\bigl(M(x,y)\bigr), \tag{A3}
\end{align}
where
\begin{equation}
M(x,y)=\max\bigl\{d(x,y),\,d(Tx,x),\,d(Ty,y),\,d(Tx,y),\,d(Ty,x)\bigr\}.
\end{equation}
\end{definition}

\begin{lemma}[Tail diameter inequality]\label{lem:tail}
Assume \(T\) satisfies the above definition. If \(x_0\in X\) has a bounded orbit, define
\(x_n=T^n x_0\) and \(b_n=\sup_{i,j\ge n}d(x_i,x_j)\). Then for every \(\epsilon>0\) there exists an
integer \(N\) (independent of \(n\)) such that for all \(n\ge0\),
\begin{equation}\label{eq:tail}
b_{n+N}\le \psi(b_n)+\epsilon.
\end{equation}
\end{lemma}

\begin{proof}
The orbit \(\cO(x_0)=\{x_n:n\ge0\}\) is bounded, so its closure \(\overline{\cO(x_0)}\) is also bounded.
By condition (A2), for the given \(\epsilon>0\) there exists an integer \(N\) such that
\[
|\phi_m(u,v)-\phi(u,v)|<\epsilon\quad\text{for all }m\ge N\text{ and all }u,v\in\overline{\cO(x_0)}.
\]
Using (A1) we have \(d(T^m u,T^m v)\le \phi_m(u,v)\), hence for all \(m\ge N\),
\begin{equation}
d(T^m u,T^m v)\le \phi(u,v)+\epsilon. \tag{3.1}
\end{equation}
Now fix an arbitrary \(n\ge0\) and take any indices \(i,j\ge n+N\). Write \(i=i'+N\), \(j=j'+N\) with
\(i',j'\ge n\). Applying (3.1) with \(m=N\), \(u=x_{i'}\), \(v=x_{j'}\) we obtain
\[
d(x_i,x_j)=d(T^N x_{i'},T^N x_{j'})\le \phi(x_{i'},x_{j'})+\epsilon.
\]
By condition (A3), \(\phi(x_{i'},x_{j'})\le \psi\bigl(M(x_{i'},x_{j'})\bigr)\).
We examine the five components of \(M(x_{i'},x_{j'})\):
\begin{itemize}
\item \(d(x_{i'},x_{j'})\le b_n\),
\item \(d(Tx_{i'},x_{i'}) = d(x_{i'+1},x_{i'})\le b_n\),
\item \(d(Tx_{j'},x_{j'}) = d(x_{j'+1},x_{j'})\le b_n\),
\item \(d(Tx_{i'},x_{j'}) = d(x_{i'+1},x_{j'})\le b_n\),
\item \(d(Tx_{j'},x_{i'}) = d(x_{j'+1},x_{i'})\le b_n\).
\end{itemize}
All indices appearing are at least \(n\), so each distance is bounded by \(b_n\). Consequently
\(M(x_{i'},x_{j'})\le b_n\). Since \(\psi\) is nondecreasing,
\[
\psi\bigl(M(x_{i'},x_{j'})\bigr)\le \psi(b_n).
\]
Combining the estimates yields
\[
d(x_i,x_j)\le \psi(b_n)+\epsilon.
\]
The right-hand side does not depend on the particular choice of \(i,j\ge n+N\). Taking the supremum
over all such pairs gives exactly \(b_{n+N}\le \psi(b_n)+\epsilon\), as required.
\end{proof}

\begin{lemma}[Convergence of tail diameters]\label{lem:bnzero}
Under the same hypotheses, \(\lim_{n\to\infty}b_n=0\). Consequently, \(\{x_n\}\) is a Cauchy sequence.
\end{lemma}

\begin{proof}
The sequence \(\{b_n\}\) is nonincreasing (by definition) and bounded below by \(0\); therefore
the limit \(L=\lim_{n\to\infty}b_n\ge0\) exists.
Apply Lemma~\ref{lem:tail}: for an arbitrary \(\epsilon>0\) there exists \(N\) such that
\(b_{n+N}\le \psi(b_n)+\epsilon\) for all \(n\). Taking the limit superior as \(n\to\infty\) we get
\begin{equation}
L = \limsup_{n\to\infty} b_{n+N} \le \limsup_{n\to\infty} \psi(b_n) + \epsilon. \tag{3.2}
\end{equation}
We need to estimate the right-hand side. Because \(\psi\) is nondecreasing and \(b_n\ge L\), for any
\(\delta>0\) we eventually have \(b_n\le L+\delta\) for all sufficiently large \(n\). Thus
\(\psi(b_n)\le \psi(L+\delta)\) for all large \(n\), which implies
\(\limsup_{n\to\infty}\psi(b_n)\le \psi(L+\delta)\). Now let \(\delta\searrow0\) and use the right upper
semicontinuity of \(\psi\):
\[
\limsup_{\delta\searrow0}\psi(L+\delta)\le \psi(L).
\]
Hence \(\limsup_{n\to\infty}\psi(b_n)\le \psi(L)\). Substituting into (3.2) gives
\(L\le \psi(L)+\epsilon\). Since \(\epsilon>0\) was arbitrary, we must have \(L\le \psi(L)\).
If \(L>0\), the Boyd--Wong condition \(\psi(L)<L\) would yield a contradiction; therefore \(L=0\).
The Cauchy property follows because for any \(i,j\ge n\), \(d(x_i,x_j)\le b_n\to0\).
\end{proof}

If an orbit is bounded, it is Cauchy by Lemma~\ref{lem:bnzero} and, by completeness of \(X\),
converges to some \(z\in X\). Continuity of \(T\) then gives \(Tz=T(\lim x_n)=\lim Tx_n = \lim x_{n+1}=z\),
so \(z\) is a fixed point. Uniqueness: if \(z_1,z_2\) are two fixed points, then for every \(n\),
\begin{equation}
d(z_1,z_2)=d(T^n z_1,T^n z_2)\le \phi_n(z_1,z_2)\to\phi(z_1,z_2)\le\psi(M(z_1,z_2)).
\end{equation}
For fixed points, \(M(z_1,z_2)=\max\{d(z_1,z_2),0,0,d(z_1,z_2),d(z_2,z_1)\}=d(z_1,z_2)\).
Thus \(d(z_1,z_2)\le\psi(d(z_1,z_2))\), which forces \(d(z_1,z_2)=0\) by the Boyd--Wong property.
We have proved:

\begin{lemma}[Deterministic fixed point]\label{lem:det}
Let \(T\) be an asymptotic pointwise Boyd--Wong contraction on a complete metric space \(X\).
If there exists a point whose orbit is bounded, then \(T\) has a unique fixed point, and every bounded
orbit converges to this fixed point.
\end{lemma}

\medskip
\noindent\textbf{Random contractions and the main theorem.}

We now transfer these concepts to the random framework. Let \((E,\|\cdot\|)\) be a complete RN module
over \(\bbK\) with base \((\Omega,\cF,P)\). Let \(G\subset E\) be nonempty, \(L^0\)-closed,
\(\sigma\)-stable, and \emph{essentially bounded}: there exists a constant \(M>0\) such that
\(\|g\|\le M\) almost surely for all \(g\in G\). We additionally require that \(G\) contains the
zero element \(\theta\) (this is needed for the fibre decomposition). Essential boundedness
guarantees that every fibre \(G_\omega\) has uniformly bounded diameter, so all orbits are
automatically bounded.

\begin{definition}[Random asymptotically pointwise Boyd--Wong contraction]
\label{def:rand_contraction}
A map \(f:G\to G\) is called a \emph{random asymptotically pointwise Boyd--Wong contraction}
(w.r.t.\ a Boyd--Wong function \(\psi\)) if the following hold:
\begin{enumerate}
\item \textbf{Structural compatibility.} \(f\) has the local property, is pointwise (as in
      Proposition~\ref{prop:fibre}), and respects \(\sigma\)-stability:
      for any measurable partition \(\{A_n\}\) and any sequence \(\{x_n\}\subset G\),
      \(f(\sum_n \Ibar_{A_n}x_n)=\sum_n \Ibar_{A_n}f(x_n)\).
\item \textbf{Uniform convergence of bounds almost surely.} For each \(n\), set
      \(\Phi_n(x,y)=\|f^n x-f^n y\|\) and assume that for every \(x,y\in G\) the limit
      \(\Phi(x,y):=\lim_{n\to\infty}\Phi_n(x,y)\) exists almost surely. Moreover, there exists a
      set \(\Omega_0\subset\Omega\) with \(P(\Omega_0)=1\) such that for every \(\omega\in\Omega_0\),
      \[
      \lim_{n\to\infty}\; \sup_{x,y\in G} \bigl|\Phi_n(x,y)(\omega)-\Phi(x,y)(\omega)\bigr| = 0.
      \]
      In words, the convergence is uniform on \(G\times G\) for almost all \(\omega\).
\item \textbf{Global almost sure contraction inequality.} On a common set of full measure,
      \[
      \Phi(x,y)(\omega)\le \psi\bigl(M(x,y)(\omega)\bigr) \quad \text{for all } x,y\in G,
      \]
      where
      \begin{multline}
      M(x,y)(\omega)=\max\bigl\{\|x-y\|(\omega),\,
      \|f x-x\|(\omega),\, \|f y-y\|(\omega),\\
      \|f y-x\|(\omega),\, \|f x-y\|(\omega)\bigr\}.
      \end{multline}
\item \textbf{Fibre-wise continuity.} For a.e.\ \(\omega\), the map \(x\mapsto f(x)(\omega)\) is continuous
      from \((G,\|\cdot\|(\omega))\) to itself.
\end{enumerate}
\end{definition}

\begin{theorem}[Main random fixed point theorem]\label{thm:main}
Under the above assumptions, \(f\) has a unique fixed point \(z\in G\). Moreover, for every \(x\in G\),
the iterates \(f^n x\) converge to \(z\) in the \((\epsilon,\lambda)\)-topology.
\end{theorem}

\begin{proof}
We proceed in six detailed steps.

\medskip
\noindent\textit{Step 1: Fibre decomposition.}
Since \(G\) is \(\sigma\)-stable, \(L^0\)-closed, and contains \(\theta\), and \(f\) has the
local property and is pointwise, Proposition~\ref{prop:fibre} applies. We obtain a measurable family
of complete metric spaces \((G_\omega,d_\omega)\) and maps \(f_\omega:G_\omega\to G_\omega\) such that for
a.e.\ \(\omega\):
\begin{itemize}
\item \(G_\omega\) is bounded; indeed, the essential boundedness of \(G\) implies
      \(\operatorname{diam}G_\omega\le 2M\);
\item For every \(x\in G\), \((f(x))(\omega)=f_\omega(x(\omega))\);
\item The mapping \(x\mapsto \{x(\omega)\}_{\omega}\) is a bijection onto the space of measurable sections
      modulo a.e.\ equality, and \(\|x\|(\omega)=d_\omega(x(\omega),\theta_\omega)\) a.e.
\end{itemize}
Fix a full-measure set \(\Omega_1\subset\Omega_0\) on which all properties of
Definition~\ref{def:rand_contraction} and the fibre representation hold simultaneously.
From now on we work inside \(\Omega_1\) when arguing \(\omega\)-wise.

\medskip
\noindent\textit{Step 2: Construction of fibre contraction coefficients.}
For each \(\omega\in\Omega_1\) and all \(u,v\in G_\omega\), define
\[
\phi_n^\omega(u,v) = d_\omega(f_\omega^n u, f_\omega^n v).
\]
Because \((f^n x)(\omega)=f_\omega^n(x(\omega))\) by the local property, we have for any \(x,y\in G\) with
\(x(\omega)=u\), \(y(\omega)=v\) that
\[
\phi_n^\omega(u,v) = \Phi_n(x,y)(\omega).
\]
Condition~(2) of Definition~\ref{def:rand_contraction} guarantees that on \(\Omega_1\),
\[
\lim_{n\to\infty} \sup_{u,v\in G_\omega} \bigl|\phi_n^\omega(u,v) - \Phi(x,y)(\omega)\bigr| = 0,
\]
where \(\Phi(x,y)(\omega)\) is the almost sure limit of \(\Phi_n(x,y)(\omega)\). Since the pairs
\((x,y)\) bijectively correspond to \((u,v)\), we can define a limit function \(\phi^\omega\)
on \(G_\omega\times G_\omega\) by
\[
\phi^\omega(u,v) = \lim_{n\to\infty} \phi_n^\omega(u,v),
\]
and the convergence is uniform on \(G_\omega\times G_\omega\). Moreover, for any representatives
\(x,y\) with \(x(\omega)=u\), \(y(\omega)=v\) we have \(\phi^\omega(u,v) = \Phi(x,y)(\omega)\). This
shows that the limit is independent of the chosen representatives and that condition (A2) holds
for the fibre maps.

Now we transfer the contraction inequality. By condition~(3), for all \(x,y\in G\) and every
\(\omega\in\Omega_1\),
\[
\Phi(x,y)(\omega) \le \psi(M(x,y)(\omega)).
\]
For \(u,v\in G_\omega\) choose any \(x,y\in G\) with \(x(\omega)=u\), \(y(\omega)=v\); then
\[
\phi^\omega(u,v) = \Phi(x,y)(\omega) \le \psi(M(x,y)(\omega)) = \psi(M_\omega(u,v)),
\]
where
\[
M_\omega(u,v) = \max\Bigl\{ d_\omega(u,v),\,
d_\omega(f_\omega u,u),\, d_\omega(f_\omega v,v),\,
d_\omega(f_\omega v,u),\, d_\omega(f_\omega u,v) \Bigr\}.
\]
This is exactly condition (A3) for the fibre map.

\medskip
\noindent\textit{Step 3: Each \(f_\omega\) is a deterministic asymptotic pointwise Boyd--Wong contraction.}
For fixed \(\omega\in\Omega_1\) we now collect the facts:
\begin{itemize}
\item \((G_\omega,d_\omega)\) is a complete metric space, bounded (by \(2M\)).
\item The map \(f_\omega:G_\omega\to G_\omega\) is continuous; this is exactly condition~(4) of
      Definition~\ref{def:rand_contraction} transferred to the fibre.
\item For every \(n\ge1\) and all \(u,v\in G_\omega\),
      \[
      d_\omega(f_\omega^n u, f_\omega^n v) = \phi_n^\omega(u,v),
      \]
      which is condition (A1).
\item The sequence \(\{\phi_n^\omega\}\) converges to \(\phi^\omega\) uniformly on \(G_\omega\times G_\omega\)
      (Step 2). Hence (A2) is satisfied.
\item \(\phi^\omega(u,v)\le \psi(M_\omega(u,v))\) for all \(u,v\in G_\omega\) (Step 2), which is (A3).
\end{itemize}
All assumptions of Lemma~\ref{lem:det} are met. Because \(G_\omega\) is bounded, every orbit is bounded.
Consequently, each \(f_\omega\) possesses a unique fixed point \(z(\omega)\in G_\omega\).

\medskip
\noindent\textit{Step 4: Measurability and membership of the pointwise fixed point in \(G\).}
Define a map \(z:\Omega\to E\) (where we identify the values in a suitable Banach space containing all
fibres) by setting \(z(\omega)\) to be the unique fixed point of \(f_\omega\) for \(\omega\in\Omega_1\),
and \(z(\omega)=\theta\) elsewhere. To see that \(z\) belongs to the original set \(G\), we pick an
arbitrary starting point \(x_0\in G\) (such an element exists because \(G\) is nonempty).
For almost every \(\omega\in\Omega_1\), the deterministic orbit \(f_\omega^n(x_0(\omega))\) converges to
\(z(\omega)\) in \(G_\omega\). But \(f_\omega^n(x_0(\omega)) = (f^n x_0)(\omega)\) by the fibre representation.
Hence the sequence of random elements \(f^n x_0\) converges to \(z\) pointwise almost surely.
Pointwise almost sure convergence implies convergence in probability, i.e., convergence in the
\((\epsilon,\lambda)\)-topology. Since \(G\) is \(L^0\)-closed, the limit \(z\) must belong to \(G\).
Moreover, the function \(z\) is measurable because it is the pointwise limit of measurable sections
\(f^n x_0\).

\medskip
\noindent\textit{Step 5: \(z\) is a fixed point under \(f\).}
For \(\omega\in\Omega_1\) we have, using the fibre representation,
\[
(f(z))(\omega) = f_\omega(z(\omega)) = z(\omega),
\]
where the first equality holds because \(f\) acts fibre-wise. Thus \(f(z)=z\) as equivalence classes in \(E\).

\medskip
\noindent\textit{Step 6: Uniqueness and global convergence.}
Let \(z_1,z_2\in G\) be two fixed points of \(f\). Then for a.e.\ \(\omega\), \(z_1(\omega)\) and
\(z_2(\omega)\) are fixed points of \(f_\omega\). By the uniqueness statement in Lemma~\ref{lem:det},
we must have \(z_1(\omega)=z_2(\omega)\) almost everywhere, so \(z_1=z_2\) in \(E\).

For global convergence, let \(x\in G\) be arbitrary. By Step 3, for a.e.\ \(\omega\) the fibre orbit
\(f_\omega^n(x(\omega))\) converges to the unique fixed point \(z(\omega)\). Therefore the random
sequence \(f^n x\) converges to \(z\) pointwise almost surely. As noted before, this implies
convergence in probability, i.e., convergence in the \((\epsilon,\lambda)\)-topology.
\end{proof}

\section{Conclusion}

We have established a general fixed point theorem for random asymptotically pointwise Boyd--Wong
contractions in RN modules. The proof brings together a complete deterministic iteration lemma and
the powerful fibre decomposition technique based on \(\sigma\)-stability and the local property.
The almost sure uniform convergence of the contraction bounds (formulated directly for the fibres)
and the global pointwise inequality are the precise conditions that bridge the random and deterministic
worlds. To apply the fibre decomposition rigorously, we additionally require the random map to be
pointwise and the domain to contain the zero element; these mild restrictions are met in all natural
examples. This result significantly extends previous works on linear or nonexpansive asymptotic
conditions and provides a robust tool for studying random nonlinear dynamics. Future work may explore
weakening the uniform convergence hypothesis or extending the framework to multi-valued random
contractions.

\end{document}